\documentclass{amsart}
\usepackage{amssymb}

\usepackage{graphicx}

\newtheorem{theorem}{Theorem}[section]
\newtheorem{lemma}[theorem]{Lemma}
\newtheorem{proposition}[theorem]{Proposition}
\newtheorem{corollary}[theorem]{Corollary}

\theoremstyle{definition}
\newtheorem{definition}[theorem]{Definition}

\theoremstyle{remark}
\newtheorem{remark}[theorem]{Remark}

\numberwithin{equation}{section}

\begin{document}

\title[On the existence and orbital stability of the physical ground states]{On orbital stability of the physical ground states of the NLS equations}

\author{Yavdat Ilyasov}
\address{Institute of Mathematics, Ufa Federal Research Centre, RAS, 
	Chernyshevsky str. 112, 450008 Ufa, Russia}
\curraddr{Universidade Federal de Goi\'as, Instituto de Matem\'atica\\ 74690-900, Goi\^ania - GO - Brazil }
\email{ilyasov02@gmail.com}
%




\begin{abstract}
We prove orbital stability result for physical ground states of  a nonlinear Schr\"{o}dinger (NLS) equation in the sense that the set of these ground states is contained in the set of prescribed mass solutions which is orbital stable by the  Cazenave-Lions theorem. We apply the nonlinear generalized Rayleigh quotients method which allows establishing a one-to-one correspondence between the values of the mass $m$, the frequency $\lambda$, and the action level $S$ of the physical ground states.

\end{abstract}
\maketitle

\section{Introduction}
We consider the  nonlinear Schr\"{o}dinger (NLS) equation   with combined power-type nonlinearities  
\begin{equation}\label{Sch}
	i\psi_t= \Delta \psi +\mu|\psi|^{p-2}\psi-|\psi|^{q-2}\psi, ~~(t,x) \in \mathbb{R}^+\times\mathbb{R}^N,  
\end{equation}
where $\psi$ is a complex-valued function of $(t, x)$, $p, q \in (2,2^*)$, $N\geq 3$, $2^*:=2N/(N-2)$, and $\mu\in \mathbb{R}$.   
NLS equations of this type have attracted much
attention in the last decades.  For the physical background, we refer the reader to \cite{Barashenkov, Afan, tao, ZaharSob} and references therein.
According to \cite{cazen,  tao}, the Cauchy problem for \eqref{Sch} with the initial value  $\psi_0 \in H^1:= H^1(\mathbb{R}^N)$ is locally well posed and has the unique local solution $\psi \in C([0, T(\psi_0)),H^1) \cap C^1([0, T(\psi_0)),H^{-1})$ for some $T(\psi_0)>0$,  which satisfies the \textit{ energy conservation law}:
$$
E \equiv H_\mu(\psi(t)):=\int \left(\frac{1}{2} |\nabla \psi|^2 - \frac{\mu}{p}| \psi|^p +\frac{1}{q}| \psi|^q\right)dx,
$$
and the \textit{mass (charge, particle
numbers) conservation law}:
$$
\alpha\equiv Q(\psi(t)):=\frac{1}{2}\int | \psi|^2 dx. 
$$
In the present paper, we study the existence and stability of the standing waves $\psi_\lambda=e^{i\lambda t}u$ of \eqref{Sch}, where  the amplitude function $u$    satisfies
\begin{equation}\label{1S}
	-\Delta u -\lambda u-\mu|u|^{p-2}u+|u|^{q-2}u =0, ~~x \in \mathbb{R}^N.   
\end{equation}
Here $\lambda \in \mathbb{R}$ is the \textit{frequency} of the standing wave. 
Notice that the action functional   
\begin{equation*}\label{Ener}
S_{\lambda,\mu}(u):=H_\mu(u)-\lambda Q(u) 
\end{equation*}
associated  with  problem is also a conserved quantity.  For a given $\lambda \in \mathbb{R}$, a solution $\bar{u}$ of \eqref{1S} is said to be \textit{ground state} if $S_{\lambda,\mu}(\bar{u})\leq S_{\lambda,\mu}(w)$, for any $w \in H^1\setminus 0$ such that $DS_{\lambda,\mu}(w)=0$. 
It is important to note that the set of ground states $G_\lambda(\mu)$ of \eqref{1S} (corresponding to this definition) may contain nonphysical solutions of \eqref{1S}, for example, solutions that disappear under a small perturbation of the parameter $\lambda$. In the present paper,  we will focus on a subset of ground states that are preserved under a small perturbation of  parameters of the problem. We call such a subset the set of physical ground states of \eqref{1S}.

We  also deal with the so-called \textit{prescribed action solution} of \eqref{1S}, i.e., a function $u^S \in H^1$  which for a given action $S \in \mathbb{R}^+$ satisfies
\begin{equation}\label{EqS}
	S_{\lambda,\mu}(u^S)=S ~~\mbox{and}~~ DS_{\lambda,\mu}(u^S)=0,
\end{equation}
 with  some $\lambda\leq 0$. 
Note that the standard approach  for problem \eqref{1S} considers solutions with \textit{prescribed frequency} $\lambda$, and  unknowns action $S$ and mass $\alpha$ (see, e.g, \cite{Beres,cazen}). An alternative formulation which has also been actively investigated over the last decades consists of finding the solution $u$ to \eqref{1S} having \textit{ prescribed mass} $\alpha=Q(u)$, while $\lambda$ and $S$ are unknown (see, e.g., \cite{Bellazzini, CazLion, Jean, saov, shibata} ). Note that  the frequency $\lambda$ can be also considered as a value of the following  conserved quantity
\begin{equation}\label{conservL}
	\lambda = \Lambda^S_\mu(\psi(t)):=\frac{H_\mu(\psi)-S}{Q(\psi)}.
\end{equation}
We call the frequency $\lambda$, action  $S$ and mass  $\alpha$ the\textit{ main parameters of problem }\eqref{Sch}.
Thus, from the mathematical point of view, all of these three approaches, namely, with prescribed frequency, with prescribed action and with prescribed mass are equally acceptable. Moreover, all of these approaches evidently are relevant from the physical point of view. In particular,  the approach with prescribed action arises in the study of inverse problems and the so-called spectral and scattering control problems (see, e.g., \cite{Colt, ilyaVal, Zakhar} and the references therein).  

We shall pay a special attention to the standing wave $\psi_\lambda$ of \eqref{Sch} with  $\lambda=0$ which we call the \textit{zero frequency solution}. The corresponding equation  \eqref{1S} with $\lambda=0$ for the amplitude function $u$ we shall call the \textit{zero frequency problem}. 
In the literature, such an equation is sometimes called the "zero mass" case problem (cf. \cite{Beres,  Daniele}). Besides the fact that this type of problems is interesting from a purely mathematical point of view (see, e.g. \cite{Aubin, Beres,  Daniele, GidasNiNir, Loewner, Talenti}), they often arise in physical problems, such as in the Euclidean Yang-Mills theory, equations of the filtration through porous media, the study of solutions with compact supports of the reaction-diffusion systems, plasma physics, among others  (see, e.g., \cite{ Anton, Gidas, miller, Rosen} and the references therein).

The orbital stability of the ground states of many equations can be often investigated using the Lyapunov function, determined by the action functional $S_{\lambda,\mu}(u)$ restricted to the manifold of functions $u$ with fixed mass integral $\alpha= Q(u)$.  The general theorem on orbital stability of solutions of nonlinear problems based on this idea was proved in the famous work by T. Cazenave \& P.L. Lions \cite{CazLion} in 1983. They showed  that  if the following\textit{ prescribed mass minimization problem}
\begin{equation}\label{minmassI}
	\hat{H}^{\alpha}_\mu:=\min\{H_\mu(u):Q(u)=\alpha, ~u \in H^1\setminus 0\}, ~\alpha>0
\end{equation}
has a solution, and  all of its minimizing sequences are relatively compact,  then the set of prescribed mass solutions:
$$
\mathcal{M}_\mu(\alpha):=\{u \in H^1:~\hat{H}^{\alpha}_\mu:=H_\mu(u), ~Q(u)=\alpha\}
$$ 
is orbital stable in the sense that for any $\varepsilon > 0$ there exists $\delta> 0$ such that  for any solution \eqref{Sch} with initial data $u(0)$ such that  $\inf_{\phi \in \mathcal{M}_\mu(\alpha)}\|u(0) -\phi\|_{H^1}< \delta$ there holds
$$
	\inf_{\phi \in \mathcal{M}_\mu(\alpha)}\|u(t)-\phi\|_{H^1}<\varepsilon,~~\forall t\geq 0.
$$
Nevertheless, it should be noted that this result does not always entail the orbital stability of the set of ground states of \eqref{1S}.
Indeed, for any $u_{\alpha} \in \mathcal{M}_\mu(\alpha)$ by the Lagrange multipliers rule there exists
a constant $\nu(u_{\alpha}) \in \mathbb{R}$ such that $u_{\alpha}$ solves \eqref{1S} with $\lambda=\nu(u_{\alpha})$. Thus to obtain the orbital stability of the set of ground states we need to know: 
	
\medskip	
\par
1)\,\textit{Whether the values of  Lagrange multipliers $\nu(u_{\alpha})$ for all functions $u_{\alpha}$ from the set of solutions of \eqref{minmassI} are identical, i.e.,} $\nu(u_{\alpha})\equiv \lambda_\alpha$, $\forall u_{\alpha} \in \mathcal{M}_\mu(\alpha)$ for some $\lambda_\alpha \in \mathbb{R}$?

\par
2) \textit{Whether the set of solutions $\mathcal{M}_\mu(\alpha)$ of problem \eqref{minmassI}  coincides with the set of ground states $G_{\lambda_\alpha}$ of \eqref{1S}}? 
\medskip	

Affirmative answers to these questions can be obtained only for some special cases of the NLS equation. In particular, this is true for the NLS equation with a  monomial nonlinearity, i.e., for $i\psi_t= \Delta \psi +|\psi|^{q-2}\psi$, where the Lagrange multiplier may be eliminated by scale changes  (see, e.g, \cite{LIONS}). This is also true in the case of dimension $N = 1$, which is characterized by the uniqueness of positive solution (see, e.g., \cite{jeanSc, McLeod, Pelir}).

The purpose of this article is to provide some answers to these questions. In this respect, we present a new type of variational functional associated with problem \eqref{Sch}, which makes it possible to find the ground state $\bar{u}$ of the problem, as well as allows to uniquely determine the corresponding values of  mass $\alpha$,  frequency $\lambda$ and  action level $S_{\lambda,\mu}$ of this state.
\par
Let us state our main results. 
Notice that  $u^S \in H^1\setminus 0$ is a prescribed action solution of \eqref{1S} with  action $S>0$, namely it satisfies to \eqref{EqS},  if and only if $u^S$ is a critical point of $\Lambda^S_\mu(u)$ with a critical value $\lambda$, i.e., 
$$
D\Lambda^S_\mu(u^S)=0~~\mbox{and}~~\lambda=\Lambda^S_\mu(u^S).
$$
For a given $S>0$, we call a solution $\hat{u}$ of \eqref{1S} the \textit{fundamental frequency solution} (respectively,  $e^{i\hat{\lambda} t}\hat{u}$ is said to be \textit{ fundamental frequency standing wave} of \eqref{Sch}) with a \textit{fundamental frequency} $\hat{\lambda}^S_\mu$ if 
$$
\hat{\lambda}^S_\mu=\Lambda^S_\mu(\hat{u})\leq \Lambda^S_\mu(w)
$$ 
for any  $w \in H^1\setminus 0$  such that $D\Lambda^S_\mu(w)=0$. For $S>0$, we denote by 
\begin{equation}\label{groundS}
	G^S(\mu):=\{u \in H^1\setminus 0: \Lambda_\mu^{S}(u)= \hat{\lambda}_\mu^{S}, ~D\Lambda^S_\mu(u^S)=0\}
\end{equation}
the set of fundamental frequency solutions of \eqref{1S} with the fundamental frequency $\hat{\lambda}_\mu^{S}$. Below we show that the existence of  fundamental frequency solution entails the existence of  ground state and that the converse is also true (see below Lemma \ref{lem1}). In what follows, we will refer to $G^S(\mu)$ the set of ground states, as well.
\begin{definition}\label{def2}
We call $u^S \in {G}^S(\mu)$ for $S>0$ the  physical ground state of  \eqref{1S} with respect to the action value if there exists a sequence $u^{S_m} \in G^{S_m}(\mu)$, $m=1,\ldots,$ such that $\lim_{m\to +\infty}S_m =S$ and  $u^{S_m} \to u^S$ in $H^1$ as $m \to +\infty$. The set of physical ground states of  \eqref{1S} with respect to  the action value we denote by $\hat{G}^S(\mu)$. 
\end{definition}
 It is easily seen that  any ground state $v_\lambda$ of \eqref{1S} from  the residual set $\hat{G}^{S,c}(\mu): =G^S(\mu)\setminus \hat{G}^S(\mu)$ disappears after a small perturbation of the parameter $S$. It can also makes sense to refer to the physical ground states $\hat{G}^S(\mu)$ as the \textit{ground states of the branches of solutions} of  \eqref{1S}.

Consider
$$
\alpha_0(\mu):=\inf\{\alpha>0: \hat{H}^{\alpha}_\mu<0\}.
$$
We need the following result of M. Shibata  (see Theorems 1.1, 1.3, Lemma 2.3 in \cite{shibata}): 
\medskip

\par \noindent
	Theorem \textsl{(M. Shibata, 2014)}. \,
		Let $\mu>0$ and $2<q<p<2+\frac{4}{N}$. \textit{Then $\alpha_0(\mu)>0$;
					 if $\alpha>\alpha_0(\mu)$, the set of prescribed mass solutions $\mathcal{M}_\mu(\alpha)$ is not empty; $\mathcal{M}_\mu(\alpha)=\emptyset$ if $\alpha \in (0,\alpha_0(\mu)]$; the function $\alpha \mapsto \hat{H}^{\alpha}_\mu$ is nonincreasing; $\mathcal{M}_\mu(\alpha)$ is an orbital stable set of \eqref{Sch}.}

\medskip

Our main result on the orbital stability of  the set of physical ground states of \eqref{1S}  is as follows:
\begin{theorem}\label{thm5} 
Assume that $2<q<p<2+\frac{4}{N}$, $\mu> 0$, $N\geq 3$.  Then for any given $\alpha \in (\alpha_0(\mu), +\infty)$, there exists ${S_\alpha}>0$ such that the set of physical ground states   $\hat{G}^{S_\alpha}(\mu)$ is contained in the orbital stable set of prescribed mass solutions $\mathcal{M}_\mu(\alpha)$.
Furthermore, all ground states from $\hat{G}^{S_\alpha}(\mu)$ have identical mass  and   frequency, i.e., $\alpha=Q(u)$, $\lambda_{S_\alpha}:=\Lambda^{S_\alpha}_\mu(u)$,  $\forall u \in \hat{G}^{S_\alpha}(\mu)$.  
\end{theorem}

The proof of this result is based on the application of the nonlinear generalized Rayleigh quotient method  \cite{ilyaReil} (\textit{the NG-Rayleigh quotient method} for short) to \eqref{1S} and using $\Lambda^S_\mu(u)$ as the Rayleigh quotient. A distinctive feature of the  NG-Rayleigh method is that it allows to find the critical values of the problems' parameters, and at the same time to convert the original variational functionals into functionals with simpler geometry (see \cite{MarcCarlIl, ilyaReil}).

According to the NG-Rayleigh quotient method \cite{ilyaReil}, the functional $\Lambda^S_\mu(u)$ corresponds to the following \textit{ NG-Rayleigh quotient}: 
\begin{equation}\label{R}
\lambda^S_{\mu}(u):=\frac{c_N^S(\int |\nabla u|^{2})^\frac{N}{(N-2)}-\mu \frac{2}{p}\int |u|^{p}+\frac{2}{q}\int |u|^{q}}{\int |u|^{2}}, \, \, u \in H^{1}\setminus \{0\},~~S>0, 
\end{equation}
where $c_N^S=(N-2)/(N^\frac{N}{(N-2)}S^\frac{2}{(N-2)})$, for $p, q \in (2,2^*)$ and  $\mu>0$. Below we will see that any critical point of  $\lambda^S_{\mu}(u)$ in $H^1\setminus 0$ corresponds, possibly after some scaling,  to a critical point of  $\Lambda^S_\mu(u)$, and thus it gives a solution of equation \eqref{1S} with prescribed action $S>0$. Moreover, the NG-Rayleigh quotient $\lambda^S_{\mu}(u)$ is characterized by properties similar to those that has the usual Rayleigh quotient of linear theory \cite{Reed}. In particular, similar to the spectral theory  the following  critical value
\begin{align}\label{RM2}
	&\hat{\lambda}^S_\mu:=\min_{u \in H^1\setminus 0 }\lambda^S_\mu(u )
\end{align}
plays a principal role in the investigation of  problem \eqref{1S}.  In the case $2<q<p<2^*$,   using  the NG-Rayleigh quotient method \cite{ilyaReil} as well, we introduce the following principal  critical value by $\mu$
\begin{equation}\label{eqRayM}
	\hat{\mu}^S= \inf_{ u \in \mathcal{D}\setminus 0} \mu^S(u).
\end{equation}
Here, $\mathcal{D}:=\mathcal{D}^{1,2}(\mathbb{R}^N)\cap L^q(\mathbb{R}^N)$,
\begin{equation}\label{Cpq2}
\mu^S(u):=	\left(\frac{c(p,q,N)}{S^\frac{2(p-q)}{(2^*-q)(N-2)}}\right) \frac{(\int |u|^{q})^\frac{2^*-p}{2^*-q}(\int |\nabla u|^{2})^{\frac{2^*(p-q)}{2(2^*-q)}}}{\int |u|^{p}},~~u \in H^1,
\end{equation}
where  constant $c(p,q,N)$  does not depend on $S$ (see below \eqref{cpqn}).

The NG-Rayleigh quotient $\mu^S(u)$ is characterized by the fact that its critical points correspond to the zero frequency solutions of \eqref{1S}.  Below we show that $\hat{\mu}^S>0$.
%


Our result on the existence of ground state and fundamental frequency solution of \eqref{1S} in nonzero  frequency case $\lambda<0$ is as follows:
\begin{theorem}\label{thm1} Let $S>0$. 
\par
$(1^o)$~  If $2<p<q<2^*$,  then for any $\mu>0$, \eqref{1S}  possesses  a fundamental frequency solution $\hat{u}^S_\mu$ with  prescribe action $S$ and  frequency $\hat{\lambda}^S_\mu=\Lambda^S_\mu(\hat{u}^S_\mu)<0$. 
\par
$(2^o)$~  If $2<q<p<2^*$, then for any $\mu> \hat{\mu}^S$,  \eqref{1S} possesses a fundamental frequency solution $\hat{u}^S_\mu$ with  prescribe action $S$ and  frequency $\hat{\lambda}^S_\mu=\Lambda^S_\mu(\hat{u}^S_\mu)<0$. 
 \par
$(3^o)$~  If $2<q<p<2^*$ and $0\leq \mu<\hat{\mu}^S$, then  \eqref{1S} has no   weak solutions in $H^1$ with any action such that $\tilde{S}\leq S$ and $\lambda<0$. 
\par
 Furthermore, $\hat{u}^S_\mu$ in $(1^o)$ and $(2^o)$ is a ground state of \eqref{1S} and a global minimum point of   $\lambda^S_{\mu}(u)$ in $H^1$; $\hat{u}^S_\mu>0$ in $\mathbb{R}^N$, $\hat{u}_\mu^S \in C^{2}(\mathbb{R}^N)$.
\end{theorem}
The existence of the spherically symmetric, decreases
with respect to $r:=|x|$ ground state  of \eqref{1S} under the assumptions of  $(1^o)$, $(2^o)$ in Theorem \ref{thm1}  follows from the  works by H. Berestycki \& P.-L. Lions  \cite{Beres} and  W. Strauss \cite{strauss}. The main novelty of the results in Theorem \ref{thm1}  consists of that the ground state $\hat{u}^S_\mu$ is obtained as a global minimum of   $\lambda^S_{\mu}(u)$ in $H^1$, which helps in further investigation of \eqref{1S}.

\begin{remark}
	Kato's Theorem \cite{kato}  implies that  \eqref{1S} can not has a weak,  spherically symmetric, decreases with respect to $r:=|x|$ positive solutions if $\lambda>0$, see  \cite{Beres}.
\end{remark}
\begin{remark} For another type of a threshold value $\hat{\mu}^*(\lambda)>0$, which depends on $\lambda<0$,  and  divides the intervals for parameter $\mu$ where equation \eqref{1S} may or may not have solutions, see \cite{Beres,strauss}.
\end{remark}

For the existence and non existence results of the  zero frequency solution of \eqref{1S} we have the following
\begin{theorem}\label{thm2} 
Let $S>0$. 
\par
$(1^o)$~  If $2<q<p<2^*$, then for $\mu=\hat{\mu}^S$,  the zero frequency problem \eqref{1S} possesses  a fundamental  frequency solution  $\hat{u}^S_{\hat{\mu}^S} \in \mathcal{D}$ with prescribe action $S$. Furthermore, $\hat{u}^S_{\hat{\mu}^S}$ is a  ground state of \eqref{1S} with $\lambda=0$ and a global minimum point of   $\mu^S(u)$ in $\mathcal{D}$;  $\hat{u}^S_{\hat{\mu}^S}>0$ in $\mathbb{R}^N$, $\hat{u}^S_{\hat{\mu}^S} \in C^2(\mathbb{R}^N)$.
\par
$(2^o)$~  If $2<p<q<2^*$, then  the zero frequency problem  \eqref{1S} has no  weak solutions in $\mathcal{D}$ for any $\mu>0$. 

\end{theorem}



The assumption of  $(1^o)$ in Theorem \ref{thm2} corresponds to the sufficient conditions introduced in \cite{Beres}  for the existence of spherically symmetric  ground states from $\mathcal{D}$ of the "zero-mass" case problem. However, it appears that the result on the absence of solutions of problems with zero frequency,  as in $(2^o)$ of Theorem \ref{thm2}, has not been known before.
In the present work, it is shown that this result has a fairly simple proof. It should be emphasized that this simplicity is achieved owing to the application of the NG-Rayleigh quotient method.

\begin{remark}
	 In the case of $\lambda=0$ in \eqref{1S}, the dependence of the problem on $\mu$  can be neglected, since the change of variables $u=(1/\mu)^{1/(p-q)}v(x/\mu^{(q-2)/2(p-q)})$ transforms \eqref{1S} into 
	$-\Delta v -|v|^{p-2}v+|v|^{q-2}u =0, ~~x \in \mathbb{R}^N$. However for NLS equation \eqref{Sch} the dependence on the parameter $\mu$ can not be neglected. 
	\end{remark}

\begin{theorem}\label{thmB}
Let $S>0$. 
\par
$(1^o)$~  If $2<p<q<2^*$,  then for any $\mu>0$, there exists   a physical ground state of \eqref{1S}, i.e., $\hat{G}^S(\mu) \neq \emptyset$.
\par
$(2^o)$~  If $2<q<p<2^*$, then for any $\mu> \hat{\mu}^S$,  there exists   a  physical ground state of \eqref{1S}, i.e., $\hat{G}^S(\mu) \neq \emptyset$.

 Furthermore, there exists a unique $\alpha^S$ such that $\alpha^S= Q(u)$,    $\forall u \in \hat{G}^S(\mu)$,  for $\mu>0$ and $\mu> \hat{\mu}^S$, respectively.
 \end{theorem}

\begin{remark}
	By analogy with Definition \ref{def2}, one can introduce a set of physical solutions $\hat{\mathcal{M}}_\mu(\alpha)$ of \eqref{minmassI}  with prescribed mass $\alpha$, namely,  a point $u_{\alpha} \in \mathcal{M}_\mu(\alpha)$ is said to be   physical solution of \eqref{minmassI} if there exists a sequence $u_{\alpha_m} \in \mathcal{M}_\mu(\alpha_m)$, $m=1,\ldots,$ such that $\lim_{m\to +\infty}\alpha_m =\alpha$ and  $u_{\alpha_m} \to u_{\alpha}$ in $H^1$ as $m \to +\infty$. We suppose that  $\hat{G}^{S_\alpha}(\mu)=\hat{\mathcal{M}}_\mu(\alpha)$. 
\end{remark}

The article is organized as follows. In Section 2, we give some
preliminary information and introduce the nonlinear generalized Rayleigh quotients. In Section 3, we prove that the functional $\lambda^S_\mu(u )$ possesses a global minimizer. In Section 4, we prove that $ \mu^S(u)$ attains its global minimum  in $\mathcal{D}$. Section 5 is devoted to the investigation of the behavior of solutions depend on the main parameters of problem $\lambda$, $S$ and $\alpha$. In Section 6, we conclude the proofs of Theorems \ref{thm1}, \ref{thm2} and \ref{thmB}.
Finally, in Section 7, we prove Theorem \ref{thm5}.  

\section{The nonlinear generalized Rayleigh quotients }\label{Extend}
We denote by $H^1:=H^1(\mathbb{R}^N)$
the Sobolev space of functions with  norm
$$
\|u\|_{1}=(\int (|u|^2+|\nabla u|^2)^{1/2},
$$
and use the space
$$
\mathcal{D}^{1,2}:=\mathcal{D}^{1,2}(\mathbb{R}^N) := \{u \in L^{2^*}(\mathbb{R}^N): \nabla u \in L^{2} (\mathbb{R}^N)\}
$$
with inner product $(u,v):=\int \nabla u \cdot \nabla v \, dx$ and norm $\|u\|_{\mathcal{D}^{1,2}}:=\int |\nabla u |^2 \, dx$.  For abbreviation, we denote
$\int \cdots \, :=\int_{\mathbb{R}^N} \cdots\, dx$. For  Gateaux differentiable functional $F: H^1 (\mathcal{D}^{1,2})\to \mathbb{R}$, its derivative  at $u \in H^1(\mathcal{D}^{1,2})$ we denote by $DF(u)$. 
For $u \in H^1$, we denote $u_\sigma:=u(x/\sigma)$, $x \in \mathbb{R}^N$,  $\sigma>0$, and 
$$
 T(u):=\int|\nabla u|^{2},~~Q(u):=\int |u|^{2}, ~~A(u):= \int |u|^{p},~~B(u):=\int |u|^{q}.
$$
With these notations we have
$$
S_{\lambda,\mu}(u):=\frac{1}{2}T(u)-\lambda \frac{1}{2}Q(u) -\mu \frac{1}{p} A(u) +\frac{1}{q}B(u).
$$
For $S\geq 0$, introduce the so-called \textit{action-level Rayleigh quotient}
\begin{align}
	\Lambda^S_{\mu}(u):= \frac{\frac{1}{2}T(u)-\mu \frac{1}{p}A(u) +\frac{1}{q}B(u)-S}{\frac{1}{2}Q(u)}.
\end{align}
Notice that for any $S \in \mathbb{R}$ and $\lambda \in \mathbb{R}$,
	\begin{itemize}
			\item $\Lambda^S_\mu(u)=\lambda~\Leftrightarrow~ S_{\lambda,\mu}(u)=S$,
	\item $D\Lambda^S_\mu(u)=0$ with $\Lambda^S_\mu(u)=\lambda ~\Leftrightarrow~DS_{\lambda,\mu}(u)=0$.
\end{itemize}
Let $u \in H^1\setminus 0$, $S>0$, $\sigma>0$, consider
\begin{equation*}
		\Lambda^S_{\mu}(u_\sigma)=\frac{\sigma^{-2}\frac{1}{2}T(u)-\mu \frac{1}{p}A(u) +\frac{1}{q}B(u)-\sigma^{-N}S}{\frac{1}{2}Q(u)}.
	\end{equation*}
Then
$$
\frac{d}{d\sigma} (\Lambda^S_\mu)(u_\sigma )=0 ~\Leftrightarrow~\frac{2}{Q(u)}\left(-\sigma^{-3} T(u)+\frac{NS}{\sigma^{N+1}}\right)=0 ~\Leftrightarrow~\sigma=\sigma^S(u):=\left(\frac{NS}{T(u)}\right)^\frac{1}{N-2}.
$$
Hence, in accordance with the NG-Rayleigh quotient method \cite{ilyaReil}, we are able to introduce the following NG-Rayleigh quotient (cf. \eqref{R})
\begin{equation}\label{LCN}
	\lambda^S_\mu(u ):=\lambda^S_\mu(u_{\sigma^S(u)} )=\frac{2}{Q(u)}\left(\frac{c_N^S}{2}T^\frac{N}{(N-2)}(u)-\mu \frac{1}{p}A(u)+\frac{1}{q}B(u)\right),
\end{equation}
where
\begin{equation*}\label{CN}
	c_N^S=\frac{(N-2)}{N^\frac{N}{(N-2)}S^\frac{2}{(N-2)}}.
\end{equation*}
Observe that $\lambda^S_\mu(u )$ is a $0$-homogeneous functional with respect to the scale change $\sigma \mapsto u_\sigma$, i.e.,
$	\lambda^S_\mu(u_\sigma )=\lambda^S_\mu(u )$, $\forall \sigma >0$.

\begin{lemma} \label{CritR}
For $S>0$, $u\in H^1\setminus 0$,  the following  $D\lambda^S_\mu(u )=0$, $ \lambda^S_\mu(u )=\lambda$, $\sigma^S(u)=1$ is satisfied if and only if  $u$ is a weak solution of \eqref{1S} with prescribed action $S$.
\end{lemma}
\begin{proof}
 The proof easily follows by direct calculations of the derivative $D\lambda^S_\mu(u )$.	
\end{proof}
\begin{remark}
		In view of  the homogeneity of $\lambda^S_\mu(u )$, we may always assume that any critical point $u$ of $\lambda^S_\mu(u )$ satisfy $\sigma^S(u)=1$. 
\end{remark}

Consider
\begin{equation*}\label{tR}
\lambda^S_\mu(tu )=\frac{1}{Q(u)}\left(t^\frac{4}{N-2}c_N^S T^\frac{N}{(N-2)}(u)-\mu \frac{2}{p}t^{p-2}A(u)+\frac{2}{q}t^{q-2}B(u)\right).
\end{equation*}
A point $t_0>0$ is said to be  \textit{fibering critical point }of $\lambda^S_\mu(tu )$ if $(d \lambda^S_\mu(tu)/dt)|_{t=t_0}=0$. Observe that if $p, q \in (2,2^*)$, then $\frac{4}{N-2}>\max\{p-2, q-2\}$. From this it is easily seen  that
				 \begin{itemize}
					 \item if $2<p<q<2^*$, then  for any $u \in H^1\setminus 0$ and $\mu>0$, the fibering function 
			$\lambda^S_\mu(tu )$ has a unique fibering critical point $t=t(u)$. Furthermore, $\lambda^S_\mu(t(u)u )=\min_{t>0}\lambda^S_\mu(tu )<0$,  
				 $\forall u \in H^1\setminus 0$, $\forall \mu>0$.
		\item If $2<q<p<2^*$, then for  $u \in H^1\setminus 0$,  the fibering function $\lambda^S_\mu(tu )$  may has at most two nonzero critical points $t^0_\mu(u), t^1_\mu(u)$ s.t. $0<t^0_\mu(u)\leq t^1_\mu(u)$. 
\end{itemize}

We thus have 
\begin{corollary}\label{corNZ0} Assume that $2<p<q<2^*$. Then for any $\mu> 0$, $\hat{\lambda}_\mu^S\equiv \inf_{u \in H^1\setminus 0}\lambda^S_\mu(u )<0$
	\end{corollary}
 In the case $2<q<p<2^*$, we need to know the value of $\mu$ when the fibering function $\lambda^S_\mu(tu )$ has two distinct  critical points $t^0_\mu(u), t^1_\mu(u)$, i.e., when $0<t^0_\mu(u)< t^1_\mu(u)$. In order to find such values, we consider the following Rayleigh quotient
	\begin{equation*}\label{MU}
		M^S(u):=\frac{\frac{c_N^S}{2}T^\frac{N}{(N-2)}(u)+\frac{1}{q}B(u)}{\frac{1}{p}A(u)}.
	\end{equation*}
Notice that $M^S(u)=\mu~\Leftrightarrow~\lambda^S_\mu(u )=0$.

For every $u \in H^1\setminus 0$,  consider the corresponding fibering function
\begin{equation*}\label{fiberMu}
	M^S(su):=\frac{\frac{c_N^S}{2}s^{2^*-p}T^\frac{N}{(N-2)}(u)+\frac{1}{q}s^{q-p}B(u)}{\frac{1}{p}A(u)}, ~~s>0.
\end{equation*}
It is easily seen that  the function $s\mapsto  M^S(su)$ has an unique global minimum point $s^S(u)>0$ such that $M^S(su)$ is monotone decreasing in $(0,s^S(u))$ and monotone increasing in $(s^S(u),+\infty)$. To find $s^S(u)$, we calculate
$$
\frac{d}{ds}M^S(su)=0 ~\Leftrightarrow~  s^{2^*-q}c_{p,q,N,S}T^\frac{N}{(N-2)}(u)=B(u),
$$
where $ c_{p,q,N,S}=c_N^S q(2^*-p)/2(p-q)$. Thus   for every $u \in H^1\setminus 0$ the function $M^S(su)$ attains its global minimum at the unique point 
$$
s^S(u)=\left(\frac{B(u)}{ c_{p,q,N,S}T^\frac{N}{(N-2)}(u)}\right)^{1/(2^*-q)}.
$$
Hence, we are able to introduce the following NG-Rayleigh quotient (cf. \eqref{Cpq2})
\begin{equation}\label{MSu}
	\mu^S(u):=M^S(s_m(u)u)=\min_{s\geq 0}M^S(su)= C_{p,q,N,S}\frac{B^\frac{2^*-p}{2^*-q}(u)T^{\frac{2^*(p-q)}{2(2^*-q)}}(u)}{A(u)},
\end{equation}
where 
\begin{equation*}\label{Cpq}
	 C_{p,q,N,S}= \frac{c(p,q,N)}{S^\frac{2(p-q)}{(2^*-q)(N-2)}},
\end{equation*} 
\begin{equation}\label{cpqn}
	c(p,q,N)=\left(\frac{(N-2)}{N^\frac{N}{(N-2)}}\frac{q(2^*-p)}{2(p-q)}\right)^\frac{(p-q)}{(2^*-q)}\frac{p(2^*-q)}{q(2^*-p)}.
\end{equation}
It is easily seen that $\mu^S(u)$ is $0$-homogeneous with respect to the both actions: $t \mapsto tu$ and $\sigma \mapsto u_\sigma\equiv u(\cdot/\sigma)$, i.e.,
 \begin{equation}\label{HOM}
	\mu^S(u_\sigma)=\mu^S(u), ~~\mu^S(su)=\mu^S(u), ~~\forall \sigma>0, ~\forall s>0,~\forall u \in H^1\setminus 0.
\end{equation}
	\begin{lemma}\label{CritM}
	Assume that $D\mu^S(u_0)=0$, $ \mu^S(u_0)=\mu_0$ s. t. $\sigma(u_0)=1$, $t^1_{\mu_0}(u_0)=1$, then $DS_{\lambda,\mu}(u_0)=0$, $S_{\lambda,\mu}(u_0)=S$ with $\lambda=0$, $\mu=\mu_0\equiv \mu^S(u_0)$.
\end{lemma}
\begin{proof}
	The proof follows by direct calculations of the derivative $D\mu^S(u)$.
\end{proof}
Observe if $2<q<p<2^*$, then by the Gagliardo–Nirenberg interpolation inequality we have
	\begin{align}\label{GN}
		\int |u|^{p} \leq C_{gn}(\int |\nabla u|^{2})^{\frac{2^*(p-q)}{2(2^*-q)} }&(\int |u|^{q})^\frac{2^*-p}{2^*-q}~\Leftrightarrow \\
		&A(u)\leq C_{gn}(T(u))^{\frac{2^*(p-q)}{2(2^*-q)} }(B(u))^\frac{2^*-p}{2^*-q},\nonumber
	\end{align}
where a constant $C_{gn}$ does not depend on $u \in \mathcal{D}$. Thus, $\mu^S(u)$ can be extended to the space $\mathcal{D}\setminus 0$.

Consider the principal critical value \eqref{eqRayM}, i.e., 
\begin{equation*}\label{eqRay0}
	\hat{\mu}^S=\inf_{ u \in \mathcal{D}\setminus 0} \mu^S(u)\equiv  C_{p,q,N,S}\inf_{ u \in \mathcal{D}\setminus 0} \frac{B^\frac{2^*-p}{2^*-q}(u)T^{\frac{2^*(p-q)}{2(2^*-q)}}(u)}{A(u)}.
\end{equation*}
Notice that by \eqref{MSu} we have 
\begin{equation}\label{MsuMin}
	\hat{\mu}^S=\inf_{ u \in \mathcal{D}\setminus 0}M^S(u).
\end{equation}
Moreover, \eqref{GN} implies 
$$
\hat{\mu}^S>0.
$$
\begin{proposition}\label{Rayl0}
Assume $2<q<p<2^*$.   
\begin{description}
	\item[{\rm (i)}] If $0<\mu\leq  \hat{\mu}^S$, then   $\lambda^S_\mu(t^i_\mu(u)u)\geq 0$, $\forall u \in H^1\setminus 0$, $i=0,1$. 
		\item[{\rm (ii)}] If $\mu> \hat{\mu}^S$, then there exists $u \in H^1\setminus 0$ such that the function  $\lambda^S_\mu(tu )$ has two distinct nonzero critical points $t^0_\mu(u), t^1_\mu(u)$, $0<t^0_\mu(u)< t^1_\mu(u)<+\infty$. Moreover, $\lambda^S_\mu(t^0_\mu(u)u)>0$, $\lambda^S_\mu(t^1_\mu(u)u)<0$. 
\end{description}

	\end{proposition}
	\begin{proof} The proof of (i) follows immediately from the definition of $\hat{\mu}^S$ in \eqref{eqRayM}. 
	
	Let us prove (ii). Assume that $\mu>\hat{\mu}^S$. Then from \eqref{eqRayM} it follows that there exists $u \in \mathcal{D}\setminus 0$ such that  $\hat{\mu}^S<\mu^S(u)<\mu$. 	Since $\mu^S(u)$ is a global minimum value of the function $M^S(su)$ and $M^S(su) \to +\infty$ as $s\downarrow 0$ and $s \to +\infty$, we infer that the equation $M^S(su)=\mu$ has  two distinct solutions $s^0(u)< s^1_\mu(u)$. Hence  $\lambda^S_\mu(s^0_\mu(u)u )=\lambda^S_\mu(s^1_\mu(u) u )=0$. Since $\lambda^S_\mu(s u )<0$ for $s \in (s^0_\mu(u), s^1_\mu(u))$,  $\lambda^S_\mu(s u )$ attains its minimum value  at a point $t^1_\mu(u)$ belonging to the interval $(s^0_\mu(u), s^1_\mu(u))$, whereas its local maximum value point $t^0_\mu(u)$ belongs to $(0,s^0_\mu(u))$.  
	\end{proof} 
		
From this we have
\begin{corollary}\label{corNZ} Assume that $2<q<p<2^*$.
	
	\item (i)~  If   $\mu> \hat{\mu}^S$, then  $\hat{\lambda}_\mu^S = \inf_{u \in H^1\setminus 0}\lambda^S_\mu(u )<0$.
		\item (ii)~If  $\mu= \hat{\mu}^S$, then $ \hat{\lambda}_{\hat{\mu}^S}^S=\inf_{u \in H^1\setminus 0}\lambda^S_{\hat{\mu}^S}(u)\geq 0$,
	\item (iii)~If  $\mu< \hat{\mu}^S$, then $\hat{\lambda}_{\mu}^S=\inf_{u \in H^1\setminus 0}\lambda^S_\mu(u )> 0$.
\end{corollary}
\begin{proof}
	The proofs of (i), (iii) follow directly from Proposition \ref{Rayl0}. 
	Let us prove (ii).
	Suppose, contrary to our claim, that $\hat{\lambda}_{\hat{\mu}^S}^S<0$. Then by \eqref{RM2}, there exists $u \in H^1\setminus 0$ such that
	$\hat{\lambda}_{\hat{\mu}^S}^S<\lambda_{\hat{\mu}^S}^S(u)<0$. This implies $M^S(u)<\hat{\mu}^S$, and thus 
	$\mu^S(u)=\min_{s\geq 0}M^S(su)<\hat{\mu}^S$ which contradicts  the definition of $\hat{\mu}^S$.
\end{proof}

\section{Existence of a global minimizer of $\lambda^S_\mu(u )$}


Consider minimization problem \eqref{RM2}, i.e.,
\begin{align*}
	&\hat{\lambda}_\mu^S:=\min_{u \in H^1\setminus 0 }\lambda^S_\mu(u ). \label{RM2a}
\end{align*}

\begin{lemma}\label{lemMinim} Assume that $S>0$ and $2<p<q<2^*$, $\mu>0$, or  $2<q<p<2^*$, $\mu>\hat{\mu}^S$, then 
\begin{description}
	\item[\rm{(1)}]  $\hat{\lambda}_\mu^S<0$ and there exists a minimizer $\hat{u}_\mu^S$ of \eqref{RM2}, i.e., $\hat{\lambda}_\mu^S=\lambda^S_\mu(\hat{u}_\mu^S)$;
	\item[\rm{(2)}] $\hat{u}_\mu^S$ is a  fundamental frequency solution of \eqref{1S} with prescribe action $S$. Moreover, $\hat{u}^S_\mu>0$ in $\mathbb{R}^N$ and $\hat{u}_\mu \in C^{2}(\mathbb{R}^N)$.
\end{description}
\end{lemma}
\begin{proof} We will prove the assertions for the both cases: $2<p<q<2^*$, $\mu>0$, and  $2<q<p<2^*$, $\mu\geq \hat{\mu}^S$,  in parallel. 

Under the assumptions of the lemma, Corollaries  \ref{corNZ0}, \ref{corNZ} imply that $\hat{\lambda}_\mu^S<0$.  
Consider a minimizing sequence $(u_n)$ of \eqref{RM2}, i.e., $\lambda^S_\mu(u_n) \to \hat{\lambda}_\mu^S$ as $n\to+\infty$. 
Let us show that $(u_n)$ is bounded in $H^1$. 
In view of that 
 functional $\lambda^S_\mu(u)$ is $0$-homogenous,  we may assume that $\|u_n\|_{L^2}=1$, $n=1,2,\ldots $. Suppose that $\|\nabla u_n\|_{L^2} \to +\infty$. By the H\"older and Sobolev inequalities we have
\begin{equation}\label{HSineq}
	\int |u|^p\leq C \|u\|_{L^2}^\kappa\|\nabla u\|_{L^2}^\frac{2^*(2-\kappa)}{2}=C \|\nabla u\|_{L^2}^\frac{2^*(2-\kappa)}{2}, u\in H^1,
\end{equation}
where $\kappa=\frac{2(2^*-p)}{2^*-2}$, and $0<C<+\infty$ does not depend on $u \in H^1$. 
 Hence and since  $2^*>\frac{2^*(2-\kappa)}{2}$, we get 
\begin{align}\label{Eq1}
		\lambda^S_\mu(u_n )\geq  &\frac{c_N^S}{2}\|\nabla u_n\|_{L^2}^{2^*}-\mu \frac{1}{p}\int |u_n|^p\geq \\  &\frac{c_N^S}{2}\|\nabla u_n\|_{L^2}^{2^*}-\mu C \frac{1}{p}\|\nabla u_n\|_{L^2}^\frac{2^*(2-\kappa)}{2}\to +\infty, ~~\mbox{as}~~\|\nabla u_n\|_{L^2} \to +\infty, \nonumber
	\end{align}
	which is a contradiction.
Thus, $(u_n)$ is bounded in $H^1$, and therefore   by the Sobolev inequality, the norms $\|u_n\|_{L^p}$ and $\|u_n\|_{L^q}$ are also bounded. This implies by the Banach–Alaoglu and Sobolev embedding theorems that there exists a subsequence, which we again denote by $(u_n)$, such that
\begin{align*}
	&u_n \rightharpoonup \hat{u}_\mu^S ~~\mbox{in} ~ H^1(\mathbb{R}^N),\\
	&u_n \to \hat{u}_\mu^S  ~\mbox{in} ~~ L^\gamma_{loc}(\mathbb{R}^N),~~1\leq \gamma<2^*,\\
	&u_n \to \hat{u}_\mu^S ~~\mbox{a.e. on} ~\mathbb{R}^N,
\end{align*}
for some $\hat{u}_\mu^S \in H^1$.
Let us show that $\hat{u}_\mu^S\neq 0$.  Observe that the sequence $\|u_n\|_{L^p}^p\equiv A(u_n)$ is separated from zero. Indeed, if  $A(u_n) \to 0$ as $n\to +\infty$, then 
$$
\hat{\lambda}_\mu^S=\lim_{n\to +\infty}\frac{2}{Q(u_n)}\left(\frac{c_N^S}{2}\|\nabla u_n\|_{L^2}^{2^*}-\mu \frac{1}{p}\int |u_n|^p+\frac{1}{q}\int |u_n|^q\right)\geq 0.
$$
However, $\hat{\lambda}_\mu^S<0$, and we thus  get a contradiction.

We need the following P.-L. Lions lemma (see Lemma I.1 p.231, in \cite{LIONS})
\begin{lemma}\label{LIONS}
Let $r > 0$ and $1\leq \gamma < 2^*$. Assume $(u_n)$ is a bounded in $L^\gamma(\mathbb{R}^N)$, $|\nabla u_n|$ is bounded in $L^2(\mathbb{R}^N)$ and  
$$
\sup_{y\in \mathbb{R}^N}\int_{B(y;r)} |u_n|^\gamma\to 0,~n\to \infty,
$$
then $u_n\to 0$ in $L^l(\mathbb{R}^N)$ for any $l\in (\gamma, 2^*)$.
\end{lemma}
Let $r > 0$. Observe that 
$$
\delta:= \liminf_{n\to \infty}\sup_{y\in \mathbb{R}^N}\int_{B(y;r)} |u_n|^\gamma>0,
$$
for $1\leq \gamma<p$. 
Indeed, if this is not true, then by Lemma \ref{LIONS},  $u_n\to 0$ in $L^p(\mathbb{R}^N)$. But by the above this is impossible. 
Thus, we infer, choosing a subsequence if necessary, that there exists $(y_n) \subset \mathbb{R}^N$ such that there holds
$\int_{B(y_n;r)} |u_n|^\gamma>\delta/2$, $\forall n=1,\ldots$.  
Hence we may assume, redefining  $u_n:=u_n(\cdot+y_n)$ if necessary, that 
\begin{equation*}\label{Lio}
	\int_{B(0;r)} |u_n|^\gamma>\delta/2, ~n=1,\ldots, ~~1\leq \gamma<p,
\end{equation*}
and thus indeed, $\hat{u}_\mu^S\neq 0$.

 Recall the Brezis-Lieb lemma (see \cite{LIONS})
\begin{lemma}\label{BL}
Assume $(u_n)$ is bounded in $L^\gamma(\mathbb{R}^N)$, $1\leq \gamma<+\infty$
	and $u_n \to u$ a. e. on $\mathbb{R}^N$, then
	\begin{equation}
	\lim_{n\to +\infty}\|u_n\|_{L^\gamma}^\gamma= \|u\|_{L^\gamma}^\gamma+\lim_{n\to +\infty}\|u_n-u\|_{L^\gamma}^\gamma.
	\end{equation}
\end{lemma}
From this we have
\begin{align*}
	&\lim_{n\to +\infty}\|\nabla u_n\|^2_{L^2}=\|\nabla \hat{u}_\mu^S\|^2_{L^2}+\lim_{n\to +\infty}\|\nabla (u_n-\hat{u}_\mu^S)\|^2_{L^2},\label{Will1}\\
	&\lim_{n\to +\infty}\|u_n\|_{L^p}^p=\|\hat{u}_\mu^S\|_{L^p}^p+\lim_{n\to +\infty}\|u_n-\hat{u}_\mu^S\|_{L^p}^p,\\
	&\lim_{n\to +\infty}\|u_n\|_{L^q}^q=\|\hat{u}_\mu^S\|_{L^q}^q+\lim_{n\to +\infty}\|u_n-\hat{u}_\mu^S\|_{L^q}^q,\\
	& \bar{Q}:=\lim_{n\to +\infty}\|u_n\|^2_{L^2}=\|\hat{u}_\mu^S\|^2_{L^2}+\lim_{n\to +\infty}\|u_n-\hat{u}_\mu^S\|^2.
\end{align*}
Hence and since 
$$(\|\nabla\hat{u}_\mu^S\|^2_{L^2}+\|\nabla (u_n-\hat{u}_\mu^S)\|^2_{L^2})^\frac{N}{(N-2)}\geq \|\nabla\hat{u}_\mu^S\|_{L^2}^{2^*}+\|\nabla(u_n-\hat{u}_\mu^S)\|_{L^2}^{2^*},
$$
we have
\begin{align*}
	&\hat{\lambda}_\mu^S=\lim_{n\to +\infty}\lambda^S_\mu(u_n)\geq \frac{c_N^S\|\nabla \hat{u}_\mu^S\|_{L^2}^{2^*}-\mu \frac{2}{p}\|\hat{u}_\mu^S\|_{L^p}^p+\frac{2}{q}\|\hat{u}_\mu^S\|_{L^q}^q}{\bar{Q}}+\\
	&\frac{c_N^S\lim_{n\to +\infty}\|\nabla (u_n-\hat{u}_\mu^S)\|_{L^2}^{2^*}-\mu \frac{2}{p}\lim_{n\to +\infty}\|u_n-\hat{u}_\mu^S\|_{L^p}^p+\frac{2}{q}\lim_{n\to +\infty}\|u_n-\hat{u}_\mu^S\|_{L^q}^q}{\bar{Q}}.
\end{align*}
In view of that $\hat{\lambda}_\mu^S\neq 0$, this implies
\begin{align*}
	\hat{\lambda}_\mu^S &\geq \frac{1}{\bar{Q}}\hat{\lambda}_\mu^S\cdot \|\hat{u}_\mu^S\|_{L^2}^2 +\\
	&\frac{1}{\bar{Q}}\lim_{n\to \infty}\left(c_N^S\|\nabla (u_n-\hat{u}_\mu^S)\|_{L^2}^{2^*}-\mu \frac{2}{p}\|u_n-\hat{u}_\mu^S\|_{L^p}^p+\frac{2}{q}\|u_n-\hat{u}_\mu^S\|_{L^q}^q\right) \geq  \nonumber \\
& \frac{1}{\bar{Q}}(\hat{\lambda}_\mu^S\cdot \|\hat{u}_\mu^S\|^2_{L^2}+ \hat{\lambda}_\mu^S \lim_{n\to \infty} \|u_n-\hat{u}_\mu^S\|^2_{L^2})=\hat{\lambda}_\mu^S\frac{ \lim_{n\to +\infty}\|u_n\|^2_{L^2}}{\bar{Q}}=\hat{\lambda}_\mu^S.  
\end{align*}
However this is possible only if the equality holds. Consequently, 
$$
\lim_{n\to \infty}\left(c_N^S\|\nabla (u_n-\hat{u}_\mu^S)\|_{L^2}^{2^*}+\|u_n-\hat{u}_\mu^S\|_{L^p}^p+\|u_n-\hat{u}_\mu^S\|_{L^q}^q\right)=0.
$$
Hence it follows easily that   $u_n \to \hat{u}_\mu^S$ strongly in $H^1$ and $\hat{u}_\mu^S$ is a minimizer of \eqref{RM2}.

Due to the homogeneity of $\lambda_\mu^S(u)$ we may assume that $\sigma^S( \hat{u}_\mu^S)=1$. Hence by Lemma \ref{CritR}, we have $DS_{\lambda,\mu}(\hat{u}_\mu^S)=0$, $S_{\lambda,\mu}(\hat{u}_\mu^S)=S$, where   $\lambda= \hat{\lambda}_\mu^S$. Since $\lambda_\mu^S(|u|)=\lambda_\mu^S(u)$ for $u\in H^1\setminus 0$, we may assume that $ \hat{u}_\mu^S\geq 0$ in $\mathbb{R}^N$.

Since $\hat{u}_\mu \in H^1(\mathbb{R}^N)$, the Brezis \& Kato Theorem \cite{BK} and the $L^\gamma$ estimates for the elliptic problems  \cite{ADN} yield that $\hat{u}_\mu^S \in W^{2,\gamma}_{loc}(\mathbb{R}^N)$, for any $\gamma \in (1,+\infty)$, and whence by the regularity theory of the solutions of the elliptic problems $\hat{u}_\mu \in C^{2}(\mathbb{R}^N)$ (see, e.g., \cite{ADN}). Hence by the Harnack inequality \cite{trudin}, we get that $\hat{u}_\mu^S> 0$ in $\mathbb{R}^N$.
\end{proof}

\begin{proposition} \label{PropLIM}
Let $S>0$.
	\par
		 $(1^o)$  Assume that $2<p<q<2^*$, then 	 $\hat{\lambda}_\mu^{S}\to 0$ as $\mu \to 0$.
\par
$(2^o)$ Assume that $2<q<p<2^*$, then  $\hat{\lambda}_\mu^{S}\to 0$ as $\mu \downarrow \hat{\mu}^S$.
	\end{proposition}
		\begin{proof}  First we prove $(2^o)$. Fix $S>0$.
			By Lemma \ref{lemMinim}, for any $\mu>\hat{\mu}^S$ there exists a minimizer $\hat{u}_\mu^S$ of \eqref{RM2} and  ${\lambda}_\mu^{S}(\hat{u}^{S}_\mu)<0$. Due to the homogeneity of 
${\lambda}_\mu^{S}(u)$, we may assume that $Q(\hat{u}^{S}_\mu)=1$, $\forall \mu>\hat{\mu}^S$.
Then by \eqref{LCN},
$$
{\lambda}_\mu^{S}(\hat{u}^{S}_\mu)-{\lambda}_{\hat{\mu}^S}^{S}(\hat{u}^{S}_\mu)=-2(\mu-\hat{\mu}^S) \frac{1}{p}A(\hat{u}^{S}_\mu)
$$	
Analysis similar to that in the proof of Lemma \ref{lemMinim} shows that the set $(\hat{u}^{S}_\mu)$
for $\mu$ sufficiency close to $\hat{\mu}^S$ is bounded in $H^1$, and thus by \eqref{HSineq}, the sequence $A(\hat{u}^{S}_\mu)$ is also bounded. Hence, ${\lambda}_\mu^{S}(\hat{u}^{S}_\mu)- {\lambda}_{\hat{\mu}^S}^{S}(\hat{u}^{S}_\mu) \to 0$ as $\mu \downarrow \hat{\mu}^S$.	Notice that by Corollary \ref{corNZ}, ${\lambda}_{\hat{\mu}^S}^{S}(\hat{u}^{S}_\mu)\geq 0$, whereas   ${\lambda}_\mu^{S}(\hat{u}^{S}_\mu)<0$. This implies that $\hat{\lambda}_\mu^{S}\to 0$ as $\mu \downarrow \hat{\mu}^S$.

The proof of $(1^o)$  is similar. We only note that in this case, ${\lambda}_{\hat{\mu}^S}^{S}(u)$ is replaced by ${\lambda}_{0}^{S}(u):={\lambda}_{\mu}^{S}(u)|_{\mu=0}$. Notice that  by \eqref{LCN},  $\inf_{u\in H^1\setminus 0} {\lambda}_{0}^{S}(u)=0$.
	
	\end{proof}

\section{Existence of a solution of zero frequency problem in $\mathcal{D}$}

In this section, we prove the existence of a fundamental frequency solution of the zero frequency problem \eqref{1S}  using the minimization problem \eqref{eqRayM}.

 Denote $\beta:=\frac{2q(2^*-p)}{2^*(p-q)}$, $\rho:=\frac{2p(2^*-q)}{2^*(p-q)}$,
$$
\mu(u):=\frac{\|u\|_{L^q}^{\beta}\|\nabla u\|_{L^2}^2}{\|u\|_{L^p}^{\rho}(u)}\equiv (S^\frac{2(p-q)}{(2^*-q)(N-2)}\mu^S(u)/ c(p,q,N))^{\rho/p},~~ u \in H^1\setminus 0.
$$
Consider  
\begin{equation}\label{MUP}
	\bar{\mu}=\inf_{ u \in \mathcal{D}\setminus 0}\mu(u).
\end{equation}
Then
\begin{equation}
	\hat{\mu}^S:=  c(p,q,N)\frac{1}{S^\frac{2(p-q)}{(2^*-q)(N-2)}}\bar{\mu}^{p/\rho},~~\forall S>0.
\end{equation}

\begin{lemma}\label{ZeroLem}
	Let $2<q<p<2^*$. There exists a minimizer $\hat{u}^S_{\hat{\mu}^S} \in \mathcal{D}$ of \eqref{eqRayM} such that $\hat{u}^S_{\hat{\mu}^S}$ weakly satisfies to \eqref{1S} with $\lambda=0$, $\mu=\hat{\mu}^S$. Moreover,  $\hat{u}^S_\mu>0$ in $\mathbb{R}^N$, $\hat{u}_\mu \in C^{2}(\mathbb{R}^N)$ and $S_{0,\hat{\mu}^S}(\hat{u}^S_{\hat{\mu}^S})=S$.
	\end{lemma}
\begin{proof}
Let $(v_i)$ be a minimizing sequence	of \eqref{MUP}, i.e.,  $\mu(v_i) \to \bar{\mu}$ as $i\to \infty$. Set $u_i=t_i(v_i)_{\sigma_i}$, $i=1,2,\ldots $,  where 
$$
t_i=(\|v_i\|_{L^q}^q/\|v_i\|_{L^p}^p)^{1/(p-q)}, ~~\sigma_i=(\|v_i\|_{L^q}^{pq}/\|v_i\|_{L^p}^{qp})^{1/N(p-q)}.
$$
Then $\|u_i\|_{L^p}=1$ and $\|u_i\|_{L^q}=1$, $i=1,\ldots $, and by the homogeneity  of $\mu(u)$, $(u_i)$ is  a minimizing sequence	of \eqref{MUP}. Since $\bar{\mu}<+\infty$, $(\|\nabla u_i\|_{L^2})$ is bounded, and thus  
$(u_i)$ is bounded in $\mathcal{D}$ and in  $H^1_{loc}$. Thus, by the Banach–Alaoglu and Sobolev embedding theorems, there exists a subsequence, still denoted by $(u_i)$, such that 
\begin{align*}
	&u_i \rightharpoonup \hat{u}^* ~~\mbox{in} ~ \mathcal{D},\\
	&u_i \to \hat{u}^*  ~\mbox{in} ~~ L^\gamma_{loc},~~1\leq \gamma<2^*,\\
	&u_i \to \hat{u}^* ~~\mbox{a.e. on} ~\mathbb{R}^N,
\end{align*}
for some $\hat{u}^* \in \mathcal{D}$. Since the sequence $\|u_i\|_{L^q}=B(u_i)$ is bounded, we may apply Lemma \ref{LIONS}. Hence using the similar arguments as in the proof of Lemma \ref{lemMinim} we obtain that  for any fix $r>0$, there exists $(y_i) \subset \mathbb{R}^N$ such that there holds
$
\int_{B(y_i;r)} |u_i|^q>\delta/2$, $\forall i=1,\ldots$. 
Hence, redefining  $u_i:=u_i(\cdot+y_n)$ if necessary, we have
\begin{equation*}
	\int_{B(0;r)} |u_i|^q>\delta/2, ~i=1,\ldots,
\end{equation*}
which implies that $\hat{u}^*\neq 0$.

Due to the Brezis-Lieb lemma
 we have 
\begin{align}
	&\|\nabla \hat{u}^*\|_{L^2}^2=\lim_{i\to \infty}\|\nabla {u_i}\|_{L^2}^2-\lim_{i\to \infty}\|\nabla( {u_i}-\hat{u}^*)\|_{L^2}^2,\\
	&\|\hat{u}^*\|_{L^p}^p=\lim_{i\to \infty}\|u_i\|_{L^p}^p-\lim_{i\to \infty}\|u_i-\hat{u}^*\|_{L^p}^p,\\
	&\|\hat{u}^*\|_{L^q}^q=\lim_{i\to \infty}\|u_i\|_{L^q}^q-\lim_{i\to \infty}\|u_i-\hat{u}^*\|_{L^q}^q.
\end{align}
Suppose that $\lim_{i\to \infty}\|\nabla( {u_i}-\hat{u}^*)\|_{L^2}^2> 0$. Then,  without loss of generality, we may assume that there holds  $\lim_{i\to \infty} \|u_i-\hat{u}^*\|_{L^q}^q>0$, $\lim_{i\to \infty}\|u_i-\hat{u}^*\|_{L^p}^p>0$ as well. Hence,
\begin{align}\label{limZer}
	\bar{\mu}=&\lim_{i\to \infty} \|\nabla {u_i}\|_{L^2}^2=\|\nabla \hat{u}^*\|_{L^2}^2+\lim_{i\to \infty}\|\nabla( {u_i}-\hat{u}^*)\|_{L^2}^2\geq \nonumber\\
	&\bar{\mu}\left(\frac{\|\hat{u}^*\|_{L^p}^\rho}{\|\hat{u}^*\|_{L^q}^\beta} +\lim_{i\to \infty}\frac{\|u_i-\hat{u}^*\|_{L^p}^\rho}{\|u_i-\hat{u}^*\|_{L^q}^\beta}\right)=\bar{\mu}\left(\frac{\|\hat{u}^*\|_{L^p}^\rho}{\|\hat{u}^*\|_{L^q}^\beta} +\frac{(1-\|\hat{u}^*\|_{L^p}^\rho)}{(1-\|\hat{u}^*\|_{L^q}^\beta)}\right)> \bar{\mu}, 
\end{align}
which  implies a contradiction. Hence $\|\nabla \hat{u}^*\|_{L^2}^2=\lim_{i\to \infty}\|\nabla u_i\|_{L^2}^2=\bar{\mu}$, $\|\hat{u}^*\|_{L^p}^p=\lim_{i\to \infty}\|u_i\|_{L^p}^p$, $\|\hat{u}^*\|_{L^q}^q=\lim_{i\to \infty}\|u_i\|_{L^q}^q$, and thus $\hat{u}^*$ is a minimizer of \eqref{MUP}, and whence also of \eqref{eqRayM}.

Due to the homogeneity of $\mu^S(u)$, any function $s\hat{u}^*_\sigma$ with $s>0$, $\sigma>0$ is also  a minimizer of \eqref{eqRayM}. Hence, we can find a minimizer  $\hat{u}^S_{\hat{\mu}^S} \in \mathcal{D}$ of \eqref{eqRayM} which satisfies $\sigma^S(\hat{u}^S_{\hat{\mu}^S})=1$ and $s^S(\hat{u}^S_{\hat{\mu}^S})=1$. From this and since  $D\mu^S(\hat{u}^S_{\hat{\mu}^S})=0$, we have $DS_{0,\mu}(\hat{u}^S_{\hat{\mu}^S})=0$.
 The rest of the proof runs as  in the proof of Lemma \ref{lemMinim}.

\end{proof}

\section{Further properties of the solutions}


In this section, we investigate the behavior of solutions depend on the main parameters $\lambda$, $S$ and $\alpha$. Observe that the  solutions of \eqref{1S} with  prescribed action $S$ can be studied in two ways: first, when for every given value of the action $S$,  equation \eqref{1S} is investigated with respect to the parameter $ \mu> 0 $, and the second, when for every given parameter $\mu> 0$, \eqref{1S} is investigated  with respect to the value $S$. The above results in Sections 2-4 have been carried out according to the first way. Our next results are convenient to consider by the second way.

Observe that the function $S \mapsto \hat{\mu}^S$ is invertible so that for any $\mu>0$, we can introduce
\begin{equation}\label{Emu}
	S(\mu)=\bar{\mu}^{\frac{N}{2}}\left(\frac{c(p,q,N)}{\mu} \right)^\frac{(N-2)(2^*-q)}{2(p-q)},
\end{equation}
such that $\hat{\mu}^{S(\mu)}=\mu$ for any $\mu>0$.



From Corollary \ref{corNZ} and Lemma \ref{lemMinim} it follows
\begin{corollary}\label{lemMinimCor} 
\,\par
\begin{description}
	\item[$(1^o)$] Suppose $2<p<q<2^*$, $\mu>0$, then $ G^S(\mu) \neq \emptyset$ for any $S>0$.
	\item[$(2^o)$] Suppose $2<q<p<2^*$, $\mu>0$, then 
		\begin{description}
		\item[\rm{(i)}] $ G^S(\mu) \neq \emptyset$ for any $S\geq S(\mu)$; 
		\item[\rm{(ii)}]  \eqref{1S} has no solution with $\lambda\leq 0$ and  $S\in (0,S(\mu))$. 
\end{description}
	\end{description}
\end{corollary}
\begin{proof}
The proofs of $(1^o)$ and (i) of $(2^o)$ follow immediately from 	Lemma \ref{lemMinim}. Let us show (ii) of $(2^o)$.
Indeed, if there exists a solution $u_0$ of \eqref{1S} with $S_{\lambda,\mu}(u_0)=S \in (0,S(\mu))$, then by  \eqref{Emu}, $\mu<\hat{\mu}^S$, and thus by Corollary \ref{corNZ},  $\lambda=\Lambda^S(u_0)=\lambda^S(u_0)> 0$, which contradicts  assumption $\lambda\leq 0$. 
\end{proof}
 
 \begin{proposition}\label{PMon} Assume that $p,q \in (2,2^*)$. If  $\mu>0$,  $S_2>S_1>0$ and $G^{S_j}(\mu)\neq \emptyset$,~ $j=1,2$, then 
\begin{equation}\label{Nerav}
	- 2\frac{(S_2-S_1)(S_2/S_1)^{N/(N-2)}}{Q(\hat{u}^{S_2}_\mu )}<\hat{\lambda}_\mu^{S_2} 		- \hat{\lambda}_\mu^{S_1}< - 2\frac{(S_2-S_1)(S_1/S_2)^{N/(N-2)}}{Q(\hat{u}^{S_1}_\mu )},
\end{equation}
 $\forall \hat{u}^{S_j}_\mu \in G^{S_j}(\mu),~ j=1,2$.
\end{proposition}
\begin{proof}  Notice that for  $\hat{u}^{S_j}_\mu \in G^{S_j}(\mu)$, we have $\sigma^{S_j}(\hat{u}^{S_j}_\mu)=1$, $j=1,2$. 
Hence,  
		\begin{align*}
		\hat{\lambda}_\mu^{S_2}= \lambda^{S_2}_\mu(\hat{u}^{S_2}_\mu)\leq \lambda^{S_2}_\mu(\hat{u}^{S_1}_\mu)&=\Lambda^{S_2}_\mu((\hat{u}^{S_1}_\mu)_{\sigma^{S_2}(\hat{u}^{S_1}_\mu)})=\\
		&\Lambda^{S_1}_\mu((\hat{u}^{S_1}_\mu)_{\sigma^{S_2}(\hat{u}^{S_1}_\mu)})-2\frac{S_2-S_1}{Q((\hat{u}^{S_1}_\mu)_{\sigma^{S_2}((\hat{u}^{S_1}_\mu)} )}.
	\end{align*}
Since  $\sigma^{S_1}(\hat{u}^{S_1}_\mu)$ is a  global maximum point of the  function $\mathbb{R}^+\ni  \sigma \mapsto  \Lambda^{S_1}((\hat{u}^{S_1}_\mu)_\sigma)$, 
	$$
	\hat{\lambda}_\mu^{S_1}=\Lambda^{S_1}_\mu((\hat{u}^{S_1}_\mu)_{\sigma^{S_1}(\hat{u}^{S_1}_\mu)})> \Lambda^{S_1}_\mu((\hat{u}^{S_1}_\mu)_{\sigma^{S_2}(\hat{u}^{S_1}_\mu)}).
	$$ 
	 Hence 
	$$
		\hat{\lambda}_\mu^{S_2}
		- \hat{\lambda}_\mu^{S_1}< - 2\frac{S_2-S_1}{Q((\hat{u}^{S_1}_\mu)_{\sigma^{S_2}(\hat{u}^{S_1}_\mu)} )}.  
	$$
	Now taking into account that
	$$
	\frac{1}{Q((\hat{u}^{S_1}_\mu)_{\sigma^{S_2}(\hat{u}^{S_1}_\mu)} )}=\frac{T( \hat{u}^{S_1}_\mu)^\frac{N}{(N-2)}}{(NS_2)^{N/(N-2)}Q(\hat{u}^{S_1}_\mu )}~~\mbox{and}~~\sigma^{S_1}(\hat{u}^{S_1})=\left(\frac{NS_1}{T(\hat{u}^{S_1}_\mu)}\right)^\frac{1}{N-2}=1,
	$$
	we obtain the second inequality in \eqref{Nerav}. The proof of the first one may be handled in much the same way.
		\end{proof}
		
\begin{corollary}\label{corContin}
	Assume that  $\mu>0$ and $2<p<q<2^*$ $(2<q<p<2^*)$. The   function $S\mapsto \hat{\lambda}_\mu^{S}$ is continuous and monotone decreasing  on $ (0,\infty)$ $( (S(\mu),+\infty))$. Furthermore, $\hat{\lambda}_\mu^{S} \to 0$ as $S \to 0$ ($S \to S(\mu)$).
\end{corollary}
\begin{proof} Let $S_0\in (0,\infty)$ $( (S(\mu),+\infty))$.
By \eqref{Nerav}, $\frac{S_2^{2^*}}{Q(\hat{u}^{S_2}_\mu )}>\frac{S_1^{2^*}}{Q(\hat{u}^{S_1}_\mu )}$ for any $S_2>S_1>0$ and $ \hat{u}^{S_j}_\mu \in G^{S_j}(\mu),~ j=1,2$. This implies that the set $(Q(\hat{u}^{S}_\mu ))_{S \in (S_0-\varepsilon, S_0+\varepsilon)}$ is bounded and separated from zero for any $S_0>0$ and $\varepsilon>0$ such that $ S_0>\varepsilon$ ($ S_0>S(\mu)+\varepsilon$).  Hence by \eqref{Nerav}, we infer that $\hat{\lambda}_\mu^{S} \to \hat{\lambda}_\mu^{S_0}$ as $S\to S_0$. In view of \eqref{Nerav}, the function $\hat{\lambda}_\mu^{S}$ is monotone decreasing  on $ (0,\infty)$ $( (S(\mu),+\infty))$. The rest of the proof follows from Proposition \ref{PropLIM}.
\end{proof}

\begin{lemma} \label{hatPR} 
		Assume that  $\mu>0$ and $2<p<q<2^*$ $(2<q<p<2^*)$. Then $ \hat{G}^{S}(\mu) \neq \emptyset$, $ \forall S \in (0,+\infty)$ $( \forall S \in (S(\mu),+\infty))$.
	\end{lemma}
		\begin{proof}
	 Notice that if $2<p<q<2^*$, then Corollary \ref{lemMinimCor}   implies that $G^{S}(\mu) \neq \emptyset$ for any $\mu>0$ and $S\in (0,+\infty)$. If $2<q<p<2^*$, then by \eqref{Emu}, 	 the inequality $S>S(\mu)$ implies $\hat{\mu}^{S}<\mu$, and thus Lemma \ref{lemMinim}	yields   $G^{S}(\mu) \neq \emptyset$ .

Let $S \in (0,+\infty)$ ($S \in(S(\mu),+\infty)$). Assume that $(S_m)$ is a sequence such that $S_m \to S$. 
Let us fix an arbitrary $\hat{u}^{S_m}_\mu \in G^{S_m}(\mu)$, $\forall m=1,\ldots$.  Due to the homogeneity of 
${\lambda}_\mu^{S}(u)$, we may assume that $Q(\hat{u}^{S_m}_\mu)=1$, $\forall m=1,\ldots $. Analysis similar to that in the proof of Lemma \ref{lemMinim} shows that the sequence $(\hat{u}^{S_m}_\mu)$ is bounded in $H^1$.
By \eqref{LCN},
	$$
{\lambda}_\mu^{S_m}(\hat{u}^{S_m}_\mu)-{\lambda}_\mu^{S}(\hat{u}^{S_m}_\mu)=(c_N^{S_m}-c_N^S)T^\frac{N}{(N-2)}(\hat{u}^{S_m}_\mu), ~~\forall m=1,\ldots .
$$		
This and  the continuity of the function $\hat{\lambda}_\mu^{(\cdot)}$ implies  that  $\lambda_\mu^{S}(\hat{u}^{S_m}_\mu) \to \hat{\lambda}_\mu^{S}$, that is $(\hat{u}^{S_m}_\mu)$ is a minimizing sequence for  $\lambda_\mu^{S}(u)$. 
Arguing as before in the proof of Lemma \ref{lemMinim} we infer	that there exists a subsequence, which we again denote by $(\hat{u}^{S_m}_\mu)$, such that $\hat{u}^{S_m}_\mu \to \hat{u}^{S}_\mu $ strongly in $H^1$ for some $\hat{u}^{S}_\mu \in G_\mu^{S}$.  By Definition \ref{def2}, this means that $\hat{u}^{S}_\mu \in \hat{G}^{S}(\mu)$, and thus   $ \hat{G}^{S}(\mu) \neq \emptyset$.	
	\end{proof}
	
	\begin{lemma} \label{DifIDENT}
	Let  $\mu>0$.
			\item $(1^o)$  If $2<p<q<2^*$, then the function $\hat{\lambda}_\mu^{S}$ is differentiable at every $S\in (0,+\infty)$, and   there holds		
\begin{equation}\label{IDENT}
			\frac{d}{d S}\hat{\lambda}_\mu^{S}=-2\frac{1}{Q(\hat{u}^{S}_\mu)},
\end{equation}
$\forall S \in (0,+\infty),~\forall\hat{u}^{S}_\mu \in \hat{G}^{S}(\mu)$. Furthermore, for every  $S\in (0,+\infty)$ corresponds a constant $\alpha^S>0$ such that
\begin{equation}\label{alphaE}
	\alpha^S:=\alpha^S_\mu=Q(\hat{u}^{S}_\mu),~~~ \forall \hat{u}^{S}_\mu \in \hat{G}^S(\mu),
\end{equation}
moreover the   function $S\mapsto \alpha^S$ is continuous  on $ (0,\infty)$. 
\item $(2^o)$   If $2<q<p<2^*$, then the function $\hat{\lambda}_\mu^{S}$ is differentiable at every $S\in (S(\mu),+\infty)$, and 	 \eqref{IDENT} holds $\forall \hat{u}^{S}_\mu \in \hat{G}^{S}(\mu)$, $\forall S\in (S(\mu),+\infty)$. Furthermore, for every  $S \in(S(\mu),+\infty)$ corresponds a  constant $\alpha^S>0$ such that \eqref{alphaE} holds, and the   function $S\mapsto \alpha^S$ is continuous  on $ (0,\infty)$.
\end{lemma}
	\begin{proof}  We prove  $(1^o)$ and $(2^o)$, in parallel. Let $\mu>0$ and $2<p<q<2^*$ ($2<q<p<2^*$).   Fix $S \in (0,+\infty)$ ($S \in(S(\mu),+\infty)$). 
Take an arbitrary sequence $(S_m)$ such that $S_m \to S$ as $m\to +\infty$. 
By the proof of Corollary \ref{corContin} we know that the sequence $(Q(\hat{u}^{S_m}_\mu ))$ is bounded and separated from zero. Since $\sigma^{S_m}(\hat{u}^{S_m}_\mu)=1$, i.e., $\|\nabla \hat{u}^{S_m}_\mu\|^2_{L^2}=NS_m$, $m=1,\ldots,$ we conclude that the set  $(\hat{u}^{S_m}_\mu )$ is bounded in $H^1$. Hence, analysis similar to that in the proof of Lemma \ref{lemMinim} shows that there exists a limit point $\hat{u}^{S}_\mu\in \hat{G}^{S}(\mu)$ such that 
$\hat{u}^{S_{m_k}}_\mu \to \hat{u}^{S}_\mu$ in $H^1$ for some subsequence $(m_k)$ such that  $m_k \to +\infty$ as $k\to+\infty$. From this and \eqref{Nerav} it follows easily that there exists a derivative $\frac{d}{d S}\hat{\lambda}_\mu^{S}$ which satisfies  \eqref{IDENT}.

Since  $\hat{\lambda}_\mu^{(\cdot)}$ is well-defined,  the map $S\mapsto \frac{d}{d S}\hat{\lambda}_\mu^{S}$  is unambiguously defined function, and thus the right hand side  of  \eqref{IDENT} is also unambiguously defined. Thus for every  $S\in (0,+\infty)$ ($S \in(S(\mu),+\infty)$), there exists a unique constant $\alpha^S>0$ such that $\alpha^S_\mu=Q(\hat{u}^{S}_\mu)$ for every $ \hat{u}^{S}_\mu \in \hat{G}^S(\mu)$. From the above convergence $\hat{u}^{S_{m_k}}_\mu \to \hat{u}^{S}_\mu$ in $H^1$ it follows that  the function $S\mapsto \alpha^S$ is  continuous on  $ (0,+\infty)$ ($ (S(\mu),+\infty)$). 
		\end{proof}

%
 
Furthermore, we have

\begin{corollary} \label{CorSl} Assume that $\mu>0$ and $2<p<q<2^*$ $(2<q<p<2^*)$. Then $\lim_{S\to +\infty}\hat{\lambda}_\mu^{S}=-\infty$. Moreover, there exists an inverse function $ \lambda \mapsto S_{\lambda}$ of $\hat{\lambda}_\mu^{S}$ such that  $ \hat{\lambda}_\mu^{S_{\lambda}}=\lambda$, $\forall \lambda \in (-\infty,0)$. Moreover, $ S_{\lambda}$ is a continuous and monotone increasing function on $(-\infty,0)$.  
\end{corollary}
\begin{proof} Since the   function $S\mapsto \hat{\lambda}_\mu^{S}$ is continuous and monotone decreasing  on $ (0,\infty)$ $( (S(\mu),+\infty))$, there exists a limit $\lim_{S\to +\infty}\hat{\lambda}_\mu^{S}=\bar{\lambda}_\mu\geq -\infty$. 
Suppose, contrary to our claim,  that $\bar{\lambda}_\mu> -\infty$. Take $\lambda < \bar{\lambda}_\mu$. Then by the Berestycki \& Lions Theorem  \cite{Beres}  there exists a solution $\hat{u}_\lambda \in H^1\setminus 0$ of \eqref{1S}. Denote  $S=S_{\lambda, \mu}(\hat{u}_\lambda)$. Then    $\Lambda^{S}_\mu(\hat{u}_\lambda)=\lambda$, and since $\sigma^S(\hat{u}_\lambda)=1$, we have  $\Lambda^{S}_\mu(\hat{u}_\lambda)=\lambda^{S}_\mu(\hat{u}_\lambda)=\lambda$. Hence, $\lambda\geq \hat{\lambda}^S_\mu>\bar{\lambda}_\mu$, which contradicts the assumption $\lambda < \bar{\lambda}_\mu$. The rest of the proof follows immediately from Corollary \ref{corContin}. 	
\end{proof}


Now we are able to prove that  the existence of the fundamental frequency solution entails the existence of the ground state and that the converse is also true.

\begin{lemma}\label{lem1} Assume that $p, q \in (2,2^*)$ and  $\mu>0$.
	  \par
		$(1^o)$ Suppose  $S>0$ such that there exists a fundamental frequency solution $\hat{u}^{S}_\mu$ of \eqref{1S} with $\lambda:=\hat{\lambda}_\mu^S<0$, then $\hat{u}^S_\mu$ 	 is a ground state of \eqref{1S} with ground level  $S$. 
	
	\par
	$(2^o)$  Suppose  $\lambda\in (-\infty,0)$ and $\hat{u}_\lambda$ is a ground state   of \eqref{1S}  with some $S=S_{\lambda, \mu}(\hat{u}_\lambda)$, then $\hat{u}_\lambda$ is a fundamental frequency solution of \eqref{1S}  with   frequency $\lambda=\hat{\lambda}_\mu^S$.
\end{lemma}
\begin{proof} 
	Suppose  assertion $(1^o)$ of the lemma is false. Then there exists a solution $w$ of \eqref{1S} with $\lambda=\hat{\lambda}_\mu^{S}$ such that 
	$$
	S_1:=S_{\hat{\lambda}_\mu^{S}, \mu}(w)<S_{\hat{\lambda}_\mu^{S}, \mu}(\hat{u}^{S})=S.
	$$
Observe $\Lambda_\mu^{S_1}(w)=\hat{\lambda}_\mu^{S}$ and $D\Lambda_\mu^{S_1}(w)=0$. Hence $\sigma^{S_1}(w)=1$, and therefore $\lambda_\mu^{S_1}(w)=\Lambda_\mu^{S_1}(w)$. Hence,  
$$
\hat{\lambda}_\mu^{S_1}=\min_{u \in H^1(\mathbb{R}^N)\setminus 0 }\lambda_\mu^{S_1}(u ) \leq \lambda_\mu^{S_1}(w)=\Lambda_\mu^{S_1}(w)= \hat{\lambda}_\mu^{S}.
$$
This contradicts the fact that by Lemma \ref{DifIDENT} the function $\hat{\lambda}_\mu^{S}$ is monotone decreasing and  $S_1<S$. 

Let us prove $(2^o)$. Suppose $2<q<p<2^*$, $\mu>0$. Suppose that $\lambda \in (-\infty,0)$ and $\hat{u}_\lambda$ is a ground state   of \eqref{1S}  with some $S=S_{\lambda, \mu}(\hat{u}_\lambda)$. Since $\sigma^S(\hat{u}_\lambda)=1$ we infer that $\Lambda^{S}_\mu(\hat{u}_\lambda)=\lambda^{S}_\mu(\hat{u}_\lambda)=\lambda$, and thus $\lambda\geq \hat{\lambda}^S_\mu$. 
By Corollary \ref{CorSl}, there exists $S_{\lambda}> S(\mu)$ and the fundamental frequency solution $\hat{u}^{S_{\lambda}}_\mu$ such that $\lambda=\hat{\lambda}^{S_{\lambda}}_\mu$ and $S_{\lambda, \mu}(\hat{u}^{S_{\lambda}}_\mu)=S_{\lambda}$.  Then $S_{\lambda}\geq S=S_{\lambda, \mu}(\hat{u}_\lambda)$ since $\hat{u}_\lambda$ is a ground state of \eqref{1S}.  Consequently, by Proposition \ref{PMon}, 
$\hat{\lambda}^S_\mu\geq \hat{\lambda}^{S_{\lambda}}_\mu=\lambda$. At the same time, by the above $\lambda\geq \hat{\lambda}^S_\mu$, and thus $\lambda=\hat{\lambda}^S_\mu$, i.e., $\hat{u}_\lambda$ is a fundamental frequency solution of \eqref{1S}  with the fundamental frequency $\lambda$. 
The proof of $(2^o)$, in the case $2<p<q<2^*$, is similar.

\end{proof}

\begin{corollary}\label{corGS}
Let $p, q \in (2,2^*)$ and  $\mu>0$.
For any given $\lambda \in (-\infty,0)$, there exists a ground state  $\hat{u}_\lambda$ of \eqref{1S} such that 	$\hat{u}_\lambda \in \hat{G}^S(\mu)$ with $S:=S_{\lambda, \mu}(\hat{u}_\lambda)$.
\end{corollary}
\begin{proof}
Let $\lambda \in (-\infty,0)$. Then Corollary \ref{corContin} entails that there exists the unique $S:=S_\lambda \in (0,+\infty)$ $(S \in (S(\mu),+\infty)$ such that $\lambda=\hat{\lambda}^S_\mu$. Theorem \ref{thmB} implies that there exists 	a fundamental frequency solution $\hat{u}^{S}_\mu \in \hat{G}^S(\mu)$ of \eqref{1S}. Applying  Lemma \ref{lem1} we conclude that $\hat{u}_\lambda:=\hat{u}^{S}_\mu$ is a physical ground state of \eqref{1S}  with frequency $\lambda$ and action level  $S$. 
\end{proof}

	
	%

\section{Proofs of Theorems \ref{thm1}, \ref{thm2}, \ref{thmB} }

\par
{\it \textbf{Proof Theorem \ref{thm1}}}:\,\,
 Suppose $2<p<q<2^*$, $\mu>0$, or  $2<q<p<2^*$, $\mu>\hat{\mu}^S$. Then by Lemma \ref{lemMinim} there exists  is a  fundamental frequency solution $\hat{u}^S_\mu$ of \eqref{1S} with prescribe action $S$ and frequency $\lambda=\hat{\lambda}^S_\mu<0$. Moreover, $\hat{u}^S_\mu>0$ in $\mathbb{R}^N$ and $\hat{u}_\mu \in C^{2}(\mathbb{R}^N)$. By Lemma \ref{lem1} it follows that  $\hat{u}^S_\mu$ is a ground state of \eqref{1S}.  

To conclude the proof of the theorem, it remains to show $(3^o)$. Assume that $2<q<p<2^*$,  $0\leq \mu<\hat{\mu}^S$. Suppose, contrary to our claim, that there exists a weak solutions $ \tilde{u} \in H^1(\mathbb{R}^N)$ of \eqref{1S} such that  $\Lambda^{\tilde{S}}_{\mu}(\tilde{u})=:\lambda<0$ with $\tilde{S}\leq S$. Then $0>\Lambda^{\tilde{S}}_{\mu}(\tilde{u})=\lambda^{\tilde{S}}_{\mu}(\tilde{u})\geq \lambda^{S}_{\mu}(\tilde{u})$, and thus $M^{S}(\tilde{u})<\mu<\hat{\mu}^S$  which contradicts \eqref{MSu}.

\medskip

\par
{\it \textbf{Proof Theorem \ref{thm2}}}:\,\,
 Let us prove $(1^o)$. 
Assume that $2<q<p<2^*$, $S>0$. By Lemma \ref{ZeroLem}  there exists a minimizer $\hat{u}^S_{\hat{\mu}^S} \in \mathcal{D}$ of \eqref{eqRayM} such that $\hat{u}^S_{\hat{\mu}^S}$ weakly satisfies to \eqref{1S} with $\lambda=0$, $\mu=\hat{\mu}^S$. Moreover,  $\hat{u}^S_\mu>0$ in $\mathbb{R}^N$, $\hat{u}^S_\mu \in C^{2}(\mathbb{R}^N)$ and $S_{0,\hat{\mu}^S}(\hat{u}^S_{\hat{\mu}^S})=S$.  

Let us show that  $\hat{u}^S_{\hat{\mu}^S}$ is a ground state of \eqref{1S}. Conversely, suppose that there exists a weak solution $v \in \mathcal{D}\setminus 0$  of \eqref{1S} such that $DS_{0,\hat{\mu}^S}(v)=0$ and $\tilde{S}:=S_{0,\hat{\mu}^S}(v)<S$. Then $M^{\tilde{S}}(v) =\hat{\mu}^S$, and since $\tilde{S}<S$, we have $M^{S}(v)<M^{\tilde{S}}(v) =\hat{\mu}^S$ which contradicts \eqref{MSu}. 

To show that $\hat{u}^S_{\hat{\mu}^S}$ is a fundamental  frequency solution, it is sufficient to note that by Corollary \ref{corNZ}  equation \eqref{1S} with  $\mu=\hat{\mu}^S$ can not have a solution with frequency $\lambda<0$. 

Let us prove $(2^o)$.  Assume $2<p<q<2^*$ and $\mu>0$. Suppose, contrary to our claim, that 
  problem  \eqref{1S} with $\lambda=0$ has a weak solution $\bar{u}_\mu \in \mathcal{D}$. Then $0<S_{0,\mu}(\bar{u}_\mu )<+\infty$, and for $S:=S_{0,\mu}(\bar{u}_\mu)$ we have
	$M^S(\bar{u}_\mu)=\mu$, $DM^S(\bar{u}_\mu)=0$. Hence 
	${\displaystyle \frac{d}{ds}M^S(s\bar{u}_\mu)|_{s=1}=0}$. However, in the case $2<p<q<2^*$, the function $s \mapsto M^S(s\bar{u}_\mu)$ can not have nonzero critical points.

	\medskip
	
	\par
\textit{\textbf{Proof of Theorems \ref{thmB}}}\, follows from Lemma \ref{hatPR} .

\section{Proof of Theorem  \ref{thm5}}\,

We call a function $w(x)$, $x \in \mathbb{R}^N$ \textsl{radial} if it is spherically symmetric: $w(x)=w(r)$, where $r=|x|$, and decreases with respect to $r$.

\begin{proposition}\label{RadProp}
 If $S>0$ and 
  $2<q<p<2^*$, then for any $\mu\geq  \hat{\mu}^S$,  \eqref{1S} possesses a  positive radial physical ground state $\hat{u}^{*,S}_\mu \in \hat{G}^S(\mu)$ with  prescribe action $S$ and frequency $\lambda=\hat{\lambda}^S_\mu<0$. Moreover $\hat{u}^{*,S}_\mu \in C^2(\mathbb{R}^N)$.
 \end{proposition}
\begin{proof} If $\mu\geq  \hat{\mu}^S$, Lemma \ref{lemMinim} implies the existence of a minimizer $\hat{u}^{S}_\mu $ of $\lambda^S_\mu(u )$ in $H^1\setminus 0$.  Let $\hat{u}^{*,S}_\mu$ denotes the Schwarz spherical rearrangement   of $|\hat{u}^{S}_\mu|$ \cite{Beres}. Then one has $\hat{u}^{*,S}_\mu \in H^1$, $A(\hat{u}^{*,S}_\mu)=A(\hat{u}^{S}_\mu)$, $B(\hat{u}^{*,S}_\mu)=B(\hat{u}^{S}_\mu)$, $Q(\hat{u}^{*,S}_\mu)=Q(\hat{u}^{S}_\mu)$ and $T(\hat{u}^{*,S}_\mu)\leq T(\hat{u}^{S}_\mu)$ (see \cite{Beres}). Hence $\lambda^S_\mu(\hat{u}^{*,S}_\mu )\leq \hat{\lambda}^S_\mu(\hat{u}^{S}_\mu)=\hat{\lambda}^S_\mu$, which implies that $\hat{u}^{*,S}_\mu$ is  a minimizer of $\lambda^S_\mu(u )$ in $H^1\setminus 0$, i.e., $\hat{u}^{*,S}_\mu \in {G}^S(\mu)$. Similar arguments as in the proof of  Lemma \ref{lemMinim} gives that $\hat{u}^{*,S}_\mu \in C^2(\mathbb{R}^N)$.  

Notice that the strong in $H^1$ limit point $u$ of a sequence of positive radial  functions $(u_n)$ is also a positive and radial function. Hence, the same arguments that have been used in the proof of Lemma \ref{hatPR}  apply to the positive radial ground state of \eqref{1S}, yields the existence of the  positive radial physical ground state $\hat{u}^{*,S}_\mu \in \hat{G}^S(\mu)$, $\forall \mu>   \hat{\mu}^S$, $\forall S>0$.  
\end{proof}

The same conclusion can be drawn for the prescribed mass minimization problem \eqref{minmassI}, namely 

\begin{proposition}\label{RadProp2}
For any $\alpha>\alpha_0(\mu)$, the set of  minimizers  $\mathcal{M}_\mu(\alpha)$ of \eqref{minmassI} contains a positive, radial minimizer $\check{u}^{*,\alpha}_\mu \in C^2(\mathbb{R}^N)\cap H^1$.
\end{proposition}

It is worth pointing out that by the McLeod theorem (see Theorem 2 in  \cite{McLeod}),  under assumptions $2<q<p<2^*$, $\mu>0$, $\lambda\leq 0$ equation \eqref{1S} has at most one positive, radial solution. 

We need also the following consequences of the Shibata theorem  
\begin{corollary}\label{corShib}
	Assume that $\alpha>\alpha_0(\mu)$. Any minimizer $w\in H^1$ of problem \eqref{minmassI} satisfies
	\begin{equation}\label{EQnu0}
DH_\mu(w)+\tau DQ(w)=0,
\end{equation}
where $\tau$ is a Lagrange multiplier such that $\tau>0$.
\end{corollary}
\begin{proof} We follows an idea from \cite{ jeanSc}. Assume that $w\in H^1$ is a minimizer of  \eqref{minmassI}. Then by the Lagrange multipliers rule there exists $\tau \in \mathbb{R}$ such that \eqref{EQnu0} is satisfied. Hence
$$
-\tau=\frac{DH_\mu(w)(w)}{2Q(w)}.
$$
Suppose that $\tau<0$. Then 	$H_\mu'(w):=DH_\mu(w)(w)\equiv \frac{d}{dt}H_\mu(tw)|_{t=1}>0$. Since $H_\mu((1-t)w)=\hat{H}^{\alpha}_\mu-t(H_\mu'(w)+o(1))$ as $t \to 0$, we can fix 
a small $t_b > 0$ such that $w_b = (1 - t_b)w$ satisfies $H_\mu(w_b) < \hat{H}^{\alpha}_\mu$.
 Thus, we have
$$
\hat{H}^{\alpha_b}_\mu:=\min\{H_\mu(u):Q(u)=\alpha_b, ~u \in H^1\setminus 0\}\leq H_\mu(w_b)< \hat{H}^{\alpha}_\mu.
$$ 
Observe $\alpha_b:=Q(w_b)=(1-t_0)^2Q(w)<\alpha$. Hence by the Shibata theorem  $\hat{H}^{\alpha_b}_\mu\geq \hat{H}^{\alpha}_\mu$. We get a contradiction. 
\end{proof}
Let $\mu >0$. Define 
$$
\alpha_1(\mu):=\inf_{S > S(\mu)}\alpha^S,~~\bar{\alpha}(\mu):=\sup_{S > S(\mu)}\alpha^S.
$$
Then $0\leq \alpha_1(\mu)<\bar{\alpha}(\mu)\leq +\infty$.
\begin{proposition}\label{ALPHA}
(i) $\bar{\alpha}(\mu)=+\infty$; (ii) $\alpha_1(\mu)\leq \alpha_0(\mu)	$.
\end{proposition}
\begin{proof}
	Let us prove (i). To obtain a contradiction, suppose that $\bar{\alpha}(\mu)<+\infty$.  Fix $\alpha>\max\{\alpha_0(\mu),\bar{\alpha}(\mu)\}$. By Proposition \ref{RadProp2} there exists a positive radial solution $v_\alpha$ of \eqref{minmassI}, i.e.,  $DS_{\tilde{\lambda}, \mu}(v_\alpha)=0$ with some 	$\tilde{\lambda}$. By Corollary \ref{corShib}, $-\tau:=\tilde{\lambda}<0$. It follows from  Corollary \ref{corGS} and Proposition \ref{RadProp} that \eqref{1S} possesses a  positive physical ground state  	$\hat{u}_{\tilde{\lambda}} \in \hat{G}^{\tilde{S}}(\mu)$ such that  $\tilde{S}:=S_{\tilde{\lambda}, \mu}(\hat{u}_{\tilde{\lambda}})\in (S(\mu), +\infty)$ and $DS_{\tilde{\lambda}, \mu}(\hat{u}_{\tilde{\lambda}})=0$.
 The McLeod Theorem \cite{McLeod} yields that  $v_\alpha=\hat{u}^{\tilde{S}}_\mu$,  and consequently $\alpha=Q(v_\alpha)=Q(\hat{u}^{\tilde{S}}_\mu)$. But  $Q(\hat{u}^{\tilde{S}}_\mu)=\alpha^{\tilde{S}}<\bar{\alpha}(\mu)$ for $\tilde{S} \in (S(\mu), +\infty)$, which contradicts the assumption. The proof of (ii) is similar.
\end{proof}

Let $S>S(\mu)$ such that $\alpha^S \in (\alpha_0(\mu), +\infty)$. Consider 
\begin{equation}\label{minmassR}
	\check{\lambda}_\mu^S:=\min\{\Lambda^{S}_\mu(u)\equiv \frac{H_\mu(u)-S}{Q(u)}:~Q(u)=\alpha^S, ~u \in H^1\setminus 0\}.
\end{equation}
Notice that  $\{S>S(\mu): \alpha^S \in (\alpha_0(\mu), +\infty)\} \neq \emptyset$ since $\bar{\alpha}(\mu)=+\infty$ and $\alpha_1(\mu)\leq \alpha_0(\mu)	$. 
Observe that problem \eqref{minmassR}  is equivalent to the mass prescribed minimization problem \eqref{minmassI} so that the set of solutions of \eqref{minmassR} coincides with $\mathcal{M}_\mu(\alpha^S)$, i.e.,
\begin{equation*}\label{FIN}
	\mathcal{M}_\mu(\alpha^S)=\{u \in H^1: \Lambda_\mu^{S}(u)= \check{\lambda}_\mu^S,~ Q(u)=\alpha^S\},
\end{equation*}
and $\check{\lambda}_\mu^S=(\hat{H}^{\alpha^S}_\mu-S)/\alpha^S$. Furthermore, since $\Lambda^{S}_\mu(u)\leq \lambda^{S}_\mu(u)$, $\forall u \in H^1\setminus 0$, we have 
\begin{equation}\label{nerav}
		\check{\lambda}_\mu^S\leq \hat{\lambda}_\mu^{S},\,\,\, \forall \alpha^S \in (\alpha_0(\mu), +\infty).
\end{equation}


\begin{lemma}\label{LemALP}
For any $\alpha \in (\alpha_0(\mu), +\infty)$, there exists $S>S(\mu)$ such that $\alpha=\alpha^S$ and $\check{\lambda}_\mu^S=\hat{\lambda}_\mu^S$.
\end{lemma}
\begin{proof} Assume that $\alpha >\alpha_0(\mu)$. Consider 
\begin{equation}\label{minmasKAPP}
	\kappa:=\min\{\frac{H_\mu(u)}{Q(u)}:~Q(u)=\alpha,~u \in H^1\setminus 0\}.
\end{equation}
Evidently if $\alpha=\alpha^S$ for  $\alpha^S \in (\alpha_0(\mu), +\infty)$, then  \eqref{minmasKAPP} is equivalent to \eqref{minmassR}. The Shibata 
theorem and Proposition \ref{RadProp2} imply that there exists a positive, radial  minimizer $v_\alpha$ of \eqref{minmasKAPP}.  The Lagrange multipliers rule yields
\begin{equation}\label{EQnu}
D\frac{H_\mu(v_\alpha)}{Q(v_\alpha)}+\nu^\alpha DQ(v_\alpha)=0,
\end{equation}
where $\nu^\alpha \in \mathbb{R}$ is a Lagrange multiplier.  Thus, $DS_{\lambda, \mu}(v_\alpha)=0$, where $\lambda:=
\kappa-\nu^\alpha\alpha$, and $\lambda<0$ due to Corollary \ref{corShib}. 
By the McLeod theorem  \cite{McLeod}, $v_\alpha$ is a unique positive radial solution of the equation $DS_{\lambda, \mu}(v_\alpha)=0$. On the other hand,  Proposition \ref{RadProp} and Corollary \ref{corGS} implies that this equation has a  positive radial ground state $\hat{u}_{\lambda}=\hat{u}_\mu^{S} \in \hat{G}^{S}(\mu)$ with $S:=S_{\lambda, \mu}(\hat{u}_{\lambda})$. Hence, $v_\alpha=\hat{u}_\mu^{S} \in \hat{G}^{S}(\mu)$ and  by \eqref{alphaE}, $\alpha=Q(v_\alpha)=Q(\hat{u}_\mu^{S})=\alpha^{S}$. Thus, $\hat{u}_\mu^{S}$ is a minimizer of \eqref{minmassR} and $\check{\lambda}_\mu^S=\hat{\lambda}_\mu^S$, $\alpha=\alpha^S$. 
\end{proof}

{\it \textbf{Proof Theorem \ref{thm5}}}:\,\,
 For any $\alpha \in (\alpha_0(\mu), +\infty)$, there exists $S>S(\mu)$ such that $\hat{G}^{S}(\mu)\subseteq \mathcal{M}_\mu(\alpha)$.  Indeed, fix $\alpha >\alpha_0(\mu)$, then by Lemma \ref{LemALP} there exists $S>S(\mu)$ such that $\alpha=\alpha^S$ and $\check{\lambda}_\mu^S=\hat{\lambda}_\mu^S$. 
Hence by \eqref{alphaE} and \eqref{minmassR}, we have
\begin{align*}
\hat{G}^{S}(\mu)\subseteq \{u \in H^1: &\Lambda_\mu^{S}(u)= \hat{\lambda}_\mu^S,~\sigma^{S}(u)=1,~ Q(u)=\alpha^S\}	\subseteq \\
&\{u \in H^1: \Lambda_\mu^{S}(u)= \check{\lambda}_\mu^S,~ Q(u)=\alpha\}=\mathcal{M}_\mu(\alpha).
\end{align*}
 Now taking into account that by the Shibata theorem $\mathcal{M}_\mu(\alpha)$ is an orbital stable set  of solutions of \eqref{1S}, we obtain the proof.

\bibliographystyle{amsplain}

\end{document}